\documentclass[12pt]{amsart}
\usepackage{amsfonts}
\usepackage{amsmath}
\usepackage{amssymb}
\usepackage[margin=1.2in]{geometry}
\newtheorem{theorem}{Theorem}[section]

\newtheorem{lemma}[theorem]{Lemma}

\renewcommand{\Re}{\operatorname{Re}}

\theoremstyle{definition}

\newtheorem{remark}[theorem]{Remark}

\makeatletter
\@namedef{subjclassname@2020}{
  \textup{2020} Mathematics Subject Classification}
\makeatother

\numberwithin{equation}{section}

\begin{document}
\title[The prime number theorem through one-sided Tauberian theorems]{The prime number theorem through one-sided Tauberian theorems}

\author[G. Debruyne]{Gregory Debruyne}
\thanks{G. Debruyne gratefully acknowledges support by a senior postdoctoral fellowship of Research--Foundation Flanders through the grant number 1249624N}
\address{G. Debruyne\\ Department of Mathematics: Analysis, Logic and Discrete Mathematics\\ Ghent University\\ Krijgslaan 281 Gebouw S8 3e verdieping\\ B 9000 Gent\\ Belgium}
\email{gregory.debruyne@ugent.be}


\begin{abstract}
In this expository article we provide an elegant proof of the one-sided Ingham-Karamata Tauberian theorem. As an application, we present a short deduction of the prime number theorem.
\end{abstract}

\maketitle

\section{Introduction}
Over the last century Tauberian theory has grown into a strong and useful tool for providing asymptotic results in surprisingly diverse areas of mathematics, ranging from number theory, probability theory, combinatorics to differential equations \cite{a-b-h-n, binghambook, katznelson, korevaar2002, korevaarbook, shubin2001}.\par
Nevertheless, the original proofs of many influential Tauberian theorems were very complicated. In 1980, Newman \cite{newman} removed some of this criticism when he provided an attractive proof of a special case of the Ingham-Karamata Tauberian theorem \cite{ingham1935,karamata1934}, from which he was able to deduce the prime number theorem. To this day, his way to the prime number theorem is considered to be one of the shortest and most beautiful available in the literature, see also the presentations by Korevaar and Zagier \cite{korevaar1982, zagier1997}. Newman's proof has drawn the attention of many mathematicians, even outside number theory, towards the Ingham-Karamata theorem, which led to the discovery of new applications. For instance, Arendt and Batty \cite{a-b} realized that the Ingham-Karamata theorem has significant applications in the theory of $C_{0}$-semigroups. This field is now one of the main driving forces behind research on the Ingham-Karamata theorem \cite{a-b-h-n, b-c-t,b-d, Sta17}. The Ingham-Karamata theorem has come to be regarded as one of the landmarks of 20th century analysis \cite{C-Qbook} because of its many applications.\par 
Newman's Tauberian theorem is concerned with Dirichlet series with \emph{bounded} coefficients and is an example of a \emph{two-sided} Tauberian theorem, meaning that one requires a \emph{Tauberian condition}\footnote{The Tauberian condition is an a priori regularity hypothesis, imposed on the object one intends to deduce asymptotic information about, without which the Tauberian theorem usually becomes false. Typical Tauberian conditions are monotonicity (one-sided) and boundedness assumptions (two-sided). In Newman's theorem, the Tauberian condition is the boundedness of coefficients of the Dirichlet series. A one-sided version would be that the coefficients are only bounded from below.} that incorporates a priori both upper and lower bounds. This also constitutes the main weakness of Newman's theorem: in applications one sometimes has to work very hard to achieve a two-sided condition as it does not always come for free. Nowadays it is widely accepted that transparent proofs for two-sided Tauberian theorems do exist, but for the original one-sided Ingham-Karamata theorem, the available proofs remain quite challenging. In fact, much recent work on the Ingham-Karamata theorem focuses on two-sided versions and whenever a one-sided version is considered, its proof is still rather technical and very delicate (see e.g. \cite{Debruyne-VindasComplexTauberians} and \cite[Prop. III. 10.2, p. 143]{korevaarbook}). Tauberian theorems with one-sided Tauberian conditions are however very important in certain fields such as number theory where counting functions of non-negative arithmetic functions are studied. The goal of this paper is to provide a transparent proof of the one-sided Ingham-Karamata theorem. Furthermore, our argument can be generalized to other one-sided Tauberian theorems, showing that in many cases, the one-sided version should not be significantly harder to prove than its two-sided counterpart. 
\par
To conclude this paper, we use our results to give a new short proof of the prime number theorem (PNT). In some way, our approach to the PNT is more direct than Newman's. Namely, we shall use our Tauberian theorem to deduce the \emph{sharp Mertens relation} from which the PNT can be established by an easy application of summation by parts. The two-sided nature of Newman's Tauberian theorem forced him not to treat the prime counting function directly, but rather a related function admitting good a priori upper and lower bounds, that is, the M\"obius function $\mu$. Using his theorem, he was able to deduce
\begin{equation} \label{eqmobius}
 \lim_{x \rightarrow \infty} \sum_{n \leq x } \frac{\mu(n)}{n} = 0.
\end{equation}
It is then well-known that the PNT can be deduced from this relation by elementary means, but the deduction, although not too hard, lies at a somewhat deeper level than mere summation by parts (see e.g. \cite[Sect. 8.1]{MV07} or \cite[Th. 3.8]{Tenenbaumbook}). Newman also discusses a second way his Tauberian theorem leads to the PNT, centered around the von Mangoldt function $\Lambda$, but then had to appeal to Chebyshev inequalities to verify the two-sided Tauberian condition of his theorem. Our approach avoids this extra step.\par

We shall use the following definition of the Fourier transform for functions $f \in L^1$: $\hat{f}(t) = \int^{\infty}_{-\infty} f(x) e^{-itx} \mathrm{d}x$. We shall use some basic results from integration and Fourier theory such as the monotone and dominated convergence theorems, the Riemann-Lebesgue lemma and Plancherel's theorem for $L^2$-functions.

\section{The one-sided Ingham-Karamata Tauberian theorem}

Our proof of the one-sided Ingham-Karamata theorem relies on a clever choice of a certain test function. We first show that such a test function in fact exists.

\begin{lemma} \label{lemexistence} Let $\varepsilon > 0$ be arbitrary. There exists a real-valued function $\phi \in L^{1} \cap L^2$ satisfying the following properties: $\int^{\infty}_{-\infty} \phi(x) \mathrm{d}x = 1$, $\phi(x) \geq 0$ for positive $x$, while $\phi(x) \leq 0$ for negative $x$, $\int^{\infty}_{-\infty} x \phi(x) \mathrm{d}x < \varepsilon$ and the Fourier transform $\hat{\phi}$ vanishes outside a compact. 
\end{lemma}
\begin{proof} Let $\phi_{0}(x) = \sin^{4}(x)/x^{4}$ whose Fourier transform is supported on $[-4,4]$. We consider $\phi_{1}(x) = x\phi_{0}(x-c) \in  L^{1} \cap L^2$, where $c$ shall be determined shortly. Then $\phi_{1} $ has compactly supported Fourier transform, clearly $x\phi_{1}(x)$ is non-negative and $\int^{\infty}_{-\infty} x \phi_{1}(x) \mathrm{d}x < \infty$. Furthermore,
\begin{equation*}
 \int^{\infty}_{-\infty} x \phi_{0}(x-c)\mathrm{d}x = c\int^{\infty}_{-\infty} \frac{\sin^{4}(x)}{x^{4}}\mathrm{d}x + \int^{\infty}_{-\infty} \frac{\sin^{4}(x)}{x^{3}}\mathrm{d}x = \frac{2 \pi c}{3}.
\end{equation*} 
We choose $c = 3/(2\pi)$. The lemma now follows upon choosing $\phi(x) = \lambda \phi_{1}(\lambda x)$ for a sufficiently large $\lambda$ ($> \varepsilon^{-1}\int^{\infty}_{-\infty} x \phi_{1}(x) \mathrm{d}x$).
\end{proof}

We show the following version of the Ingham-Karamata Tauberian theorem.

\begin{theorem} \label{thonesided} Let $S:  [0,\infty) \rightarrow \mathbb{R}$ be such that $S(x) + Mx$ is non-decreasing for a positive constant $M$. Suppose that the improper integral
\begin{equation}
 \mathcal{L}\{S;s\} := \int^{\infty}_{0} e^{-sx} S(x) \mathrm{d}x \ \ \ \text{converges for } \Re s > 0,
\end{equation} 
and that $\mathcal{L}\{S;s\}$ admits a continuous extension to $\Re s = 0$. Then 
\begin{equation}
 \lim_{x \rightarrow \infty} S(x) = 0.
\end{equation}
\end{theorem}
\begin{proof} We first show the very rough upper bound $S(x) \leq e^{\sigma x}$ for every sufficiently large $x$ (depending on $\sigma$) and all $\sigma > 0$. For this we will just need the Tauberian condition of $S$ and the convergence of the Laplace transform for $\sigma > 0$. In this case we may suppose without loss of generality that $S$ is positive and increasing. Assume that the bound does not hold. Then there exists a sequence $x_{n} \rightarrow \infty$, with $x_{n+1}-x_{n} \geq 1$ and $S(x_{n}) \geq e^{\sigma x_{n}}$. Then,
\begin{equation*}
 \int^{\infty}_{0} e^{-\sigma x} S(x) \mathrm{d}x \geq \sum_{n= 1}^{\infty} \int^{x_{n}+1}_{x_{n}} e^{-\sigma x} S(x) \mathrm{d}x \geq  e^{-\sigma} \sum_{n = 1}^{\infty} 1 = \infty.
\end{equation*}
Therefore, $S$ satisfies the required exponential bound and this shows in particular\footnote{The Tauberian condition immediately implies the lower bound $S(x) \geq -Mx + S(0)$ for $x \geq 0$.} that $S(x) e^{-\sigma x} \in L^{2}$ for every $\sigma > 0$. 

The first step of the proof is to translate the behavior of the Laplace transform of $S$ into asymptotic information on a certain convolution relation $S \ast \phi$. We shall impose the conditions of Lemma \ref{lemexistence} on $\phi$: $\phi \in L^{1} \cap L^2$, $\int^{\infty}_{-\infty} \phi(x) \mathrm{d}x = 1$, $x \phi(x)$ is positive, $\int^{\infty}_{-\infty} x \phi(x) \mathrm{d}x < \varepsilon$ and the Fourier transform of $\phi$ vanishes outside a compact interval, say $[-R,R]$. (The value of $R$ depends on $\varepsilon$.) The reason for imposing precisely these conditions will become apparent from the application of the Tauberian condition at the end of the proof. 

Let $h$ be a sufficiently large real number. Let $X = X_{h}> 0$ (which by enlarging $M$ if necessary, we may assume without loss of generality to be finite) be such that $S(x+h) + M(x+h)$ is positive on $(X,\infty)$. For convenience we extend $S(x)$ to be $0$ on the negative half-axis. By the dominated (on $(-h,X)$) and monotone  (on $(X,\infty)$) convergence theorems and Plancherel's formula (for $L^{2}$-functions), we obtain
\begin{align*}
 \int^{\infty}_{-\infty} S(x+h) \phi(x) \mathrm{d}x &= \int^{\infty}_{-h} (S(x+h)+ M(x+h)) \phi(x) \mathrm{d}x - \int^{\infty}_{-h}  M(x+h)\phi(x) \mathrm{d}x \\
&= \lim_{\sigma \rightarrow 0^{+}} \left(\int^{X}_{-h} + \int^{\infty}_{X}\right)   e^{-\sigma x} (S(x+h)+M(x+h))\phi(x) \mathrm{d}x \\
& \quad \quad - \int^{\infty}_{-h}  M(x+h)\phi(x) \mathrm{d}x\\
& = \lim_{\sigma \rightarrow 0^{+}} \int^{\infty}_{-\infty} e^{-\sigma x} S(x+h)\phi(x) \mathrm{d}x\\
& = \lim_{\sigma \rightarrow 0^{+}} \frac{1}{2\pi} \int^{R}_{-R} e^{(\sigma+it)h} \mathcal{L}\{S;\sigma + it \} \hat{\phi}(-t) \mathrm{d}t\\
& = \frac{1}{2\pi} \int^{R}_{-R} G(t) e^{iht} \hat{\phi}(-t) \mathrm{d}t,
\end{align*}
where $G$ and $\hat{\phi}$ are continuous functions. The Riemann-Lebesgue lemma ensures then
\begin{equation} \label{eqconvav}
 \lim_{h \rightarrow \infty} \int^{\infty}_{-\infty} S(x+h) \phi(x) \mathrm{d}x = 0, 
\end{equation}
for functions $\phi$ satisfying the imposed properties.\par
 We now deduce asymptotic information of $S$ from this `average'. Let $\varepsilon > 0$ be arbitrary and let $\phi$ satisfy the hypotheses from Lemma \ref{lemexistence}. Since $S(x) + Mx$ is non-decreasing, we obtain through the properties of $\phi$: 
\begin{align*}
 S(h) &= \int^{\infty}_{-\infty} (S(h) + Mh) \phi(x) \mathrm{d}x  - Mh \leq \int^{\infty}_{-\infty} (S(x+h) + M(x+h)) \phi(x) \mathrm{d}x - Mh\\
& = \int^{\infty}_{-\infty} S(x+h) \phi(x) \mathrm{d}x + M\varepsilon,
\end{align*}
where we take $h$ sufficiently large to ensure that $S(h) + Mh \geq 0$. (This guarantees that the inequality $S(x+h) + M(x+h) \leq S(h) + Mh$ remains valid for $x \leq -h$.) Letting $h \rightarrow \infty$, we obtain $\limsup_{h \rightarrow \infty} S(h) \leq M\varepsilon$. Applying the same procedure with $\phi(-x)$, we obtain $\liminf_{h \rightarrow \infty} S(h) \geq -M\varepsilon$. Since $\varepsilon$ was arbitrary, this completes the proof.
\end{proof}

\begin{remark} The simplification of the proof mainly lies in the Tauberian step at the end. Over the course of the last century, it became well-known \cite{korevaar2005, korevaar2005FR} that \eqref{eqconvav} is a suitable translation of the properties of the Laplace transform, for certain appropriate test functions $\phi$. However, in earlier work \cite{Debruyne-VindasComplexTauberians}, one usually takes $\phi$ to be non-negative. This has as consequence that one can no longer use the Tauberian condition in the same fashion as at the end of the proof. (For $x < 0$, $(S(x+h) + M(x+h))\phi(x) \leq (S(h) + M(h)) \phi(x))$, which is the wrong sign.) One is still able to show the final result however, but the argument is much more delicate. \par
Recently \cite{Debruyne-VindasComplexTauberians}, the hypotheses of continuous extension of the Laplace transform have been weakened to a minimum, to so-called \emph{local pseudofunction boundary behavior}. We mention that the proof we give above works in exactly the same way to obtain this more general result. With some small modifications one can also treat the more general one-sided Tauberian condition of \emph{slow decrease}. 
\end{remark}

\section{The prime number theorem}

In this section we deduce the prime number theorem from our one-sided Theorem \ref{thonesided} and the absence of $\zeta$-zeros on the line $\Re s = 1$. Let $\pi(x)$ be the counting function of the primes. The prime number theorem asserts that
\begin{equation}
 \lim_{x \rightarrow \infty} \frac{\pi(x)}{x/\log x} = 1.
\end{equation}
Since the function $\pi$ does not lend itself well to be treated directly with analytical methods, one usually considers the Chebyshev function $\psi(x)$ and shows
\begin{equation} \label{eqchebyshev}
 \psi(x) := \sum_{p^{k} \leq x} \log p, \ \ \ \lim_{x \rightarrow \infty} \frac{\psi(x)}{x} = 1.
\end{equation}
We are going to show the \emph{sharp Mertens relation}, that is,
\begin{equation} \label{eqsharpmertens}
 \psi_{1}(x) := \sum_{p^{k} \leq x} \frac{\log p}{p^{k}}, \ \ \ \lim_{x \rightarrow \infty} \psi_{1}(x) - \log x = -\gamma,
\end{equation}
where $\gamma$ is the Euler-Mascheroni constant $0.57...$. It is an easy exercise in summation by parts to see that \eqref{eqsharpmertens} implies the PNT. We leave the details to the reader.


\begin{remark} The M\"obius relation \eqref{eqmobius} and the sharp Mertens relation \eqref{eqsharpmertens} are only two of several expressions that are considered to be \emph{PNT equivalences}---these are expressions that can be shown to imply and are implied by the PNT by purely elementary methods, that is, without the use of complex analysis or Fourier theory. The terminology was introduced before the elementary proof of the PNT and has since then lost any logical foundation, but is convenient and remains in use today. Namely, these equivalences are nowadays being studied in the context of \emph{Beurling generalized prime numbers}. We refer to \cite[Ch. 14]{diamond-zhangbook} for a survey on the PNT equivalences in Beurling numbers and to \cite{Debruyne-VindasPNTEquivalences} for some more recent results related to the sharp Mertens relation in Beurling prime number systems. We only mention here that in the world of Beurling numbers, the PNT does not necessarily imply the sharp Mertens relation without additional assumptions.
\end{remark}

 We now come to the deduction of the sharp Mertens relation from our Tauberian theorem. We intend to apply Theorem \ref{thonesided} to $\psi_{1}(e^{x})$. Therefore, we need to compute its Laplace transform. As it shall be given in terms of the $\zeta$ function, we first recall some well-known properties of $\zeta$, that can be found in any good textbook on analytic number theory \cite{MV07,Tenenbaumbook}. The $\zeta$ function is given by
\begin{equation}\label{eqdefzeta}
  \zeta(s) : = \sum_{n= 1}^{\infty} n^{-s}= \prod_{p} \frac{1}{1-p^{-s}}, \ \ \ \Re s > 1.
\end{equation}
The $\zeta$ function admits a meromorphic continuation beyond the line $\Re s =1$, whose only singularity resides at $s = 1$. There $\zeta(s)$ admits a pole of order $1$ with residue $1$ and the constant term of the Laurent expansion of $\zeta$ at $s = 1$ is $\gamma$. Furthermore, $\zeta(s)$ does not have any zeros on the half-plane $\Re s \geq 1$. \par
Taking the logarithmic derivative of \eqref{eqdefzeta} gives
\begin{equation}
 -\frac{\zeta'(s)}{\zeta(s)} = \sum_{p} \sum_{k= 1}^{\infty}  p^{-ks}\log p , \ \ \ \Re s > 1.
\end{equation}
From this, one may deduce the Laplace transform of $\psi_{1}(e^{x})$ by summation by parts, 
\begin{equation}
  \mathcal{L}\{\psi_{1}(e^{x}); s\} = \frac{1}{s} \sum_{p^{k}} p^{-k(s+1)} \log p = -\frac{\zeta'(s+1)}{s\zeta(s+1)}, \ \ \ \Re s > 0.
\end{equation}

We note that we cannot immediately apply Theorem \ref{thonesided} to $\psi_{1}(e^{x})$ as $-\zeta'(s+1)/s\zeta(s+1)$ has a second order pole at $s = 0$. We therefore first subtract the singularities of this Laplace transform. Concretely, we shall apply Theorem \ref{thonesided} to $S(x) := \psi_{1}(e^{x}) - x + \gamma$. Since $\psi_{1}$ is non-decreasing, it clearly follows that $S(x) + x$ is non-decreasing and hence the Tauberian condition of Theorem \ref{thonesided} is fulfilled. Furthermore,
\begin{equation}
 \mathcal{L}\{S;s\} = -\frac{\zeta'(s+1)}{s\zeta(s+1)} -\frac{1}{s^{2}} + \frac{\gamma}{s}, \ \ \ \Re  s > 0,
\end{equation}
and this function admits a continuous extension on $\Re s = 0$ (because of the Laurent expansion of $\zeta$ at $s= 1$ and the absence of zeros of $\zeta$ on the line $\Re s = 1$). Therefore, all the conditions of Theorem \ref{thonesided} are fulfilled and we deduce 
\begin{equation}
 \lim_{x \rightarrow \infty} \psi(e^{x}) - x + \gamma = 0,
\end{equation}
which is equivalent to \eqref{eqsharpmertens} and from which the PNT follows.

\end{document}